\documentclass[12pt]{amsart}
\usepackage[utf8]{inputenc}
\usepackage{amstext}
\usepackage{amsmath, amssymb, amsthm}
\usepackage{amsthm}
\usepackage{amssymb}
\usepackage[a4paper, total={5.76in, 8.5in}]{geometry}
\usepackage{enumerate}
\usepackage{xcolor}
\usepackage{comment}
\usepackage{hyperref}
\usepackage{setspace}

\newtheorem{theorem}{Theorem}[section]
\newtheorem{lemma}[theorem]{Lemma}
\newtheorem{proposition}[theorem]{Proposition}
\newtheorem{corollary}[theorem]{Corollary}
\newtheorem{remark}[theorem]{Remark}

\DeclareMathOperator{\Rm}{Rm}
\DeclareMathOperator{\Ric}{Ric}
\DeclareMathOperator{\vol}{vol}

\title[Ancient solutions of Ricci flow]{Ancient solutions of Ricci flow with Type~I curvature growth}
\author{Stephen Lynch}
\author{Andoni Royo Abrego}

\address{Department of Mathematics, Imperial College London, London SW7 2AZ, United Kingdom}
\email{stephen.lynch@imperial.ac.uk}

\address{Eberhard Karls Universit\"at T\"ubingen, Fachbereich Mathematik, Auf der Morgenstelle 10, 72076 T\"{u}bingen, Germany}
\email{andoni.royo-abrego@uni-tuebingen.de}

\begin{document}

\maketitle

\begin{abstract}
    Ancient solutions of the Ricci flow arise naturally as models for singularity formation. There has been significant progress towards the classification of such solutions under natural geometric assumptions. Nonnegatively curved solutions in dimensions 2 and 3, and uniformly PIC solutions in higher dimensions are now well understood. We consider ancient solutions of arbitrary dimension which are complete and have Type~I curvature growth. We show that a simply connected, noncompact, $\kappa$-noncollapsed Type~I ancient solution with nonnegative sectional curvature necessarily splits a Euclidean factor. It follows that a $\kappa$-noncollapsed Type~I ancient solution which is weakly PIC2 is a locally symmetric space. 
\end{abstract}

\section{Introduction}

A solution of a geometric flow, such as the Ricci flow or mean curvature flow, is called ancient if it exists on a time interval of the form $(-\infty,T]$. Rescaling shows that such solutions model the flow in regions of high curvature. Therefore, determining the possible shapes of ancient solutions is a central problem in the study of singularities.

In recent years there has been significant progress towards the classification of ancient solutions to the Ricci flow and mean curvature flow under natural geometric conditions. For example, all 3-dimensional ancient solutions of Ricci flow which are nonnegatively curved and $\kappa$-noncollapsed have been classified, in \cite{Perelman, Brendle_13, Brendle-Huisken-Sinestrari, Brendle_20, Angenent-Daskalopoulos-Sesum_22, Brendle-Daskalopoulos-Sesum}. For mean curvature flow, ancient solutions in $\mathbb{R}^3$ which are convex and interior noncollapsed have been classified in \cite{Huisken-Sinestrari, Haslhofer-Hershkovits, Haslhofer, Brendle-Choi, Angenent-Daskalopoulos-Sesum_19, Angenent-Daskalopoulos-Sesum_20}. These results have been generalised and expanded upon in various directions; further references are given below. 

In \cite{Lynch}, the first-named author classified ancient solutions of mean curvature flow (and a natural class of fully nonlinear flows) which are convex and have Type~I curvature growth --- all such solutions are homothetically shrinking cylinders. The key step was to establish that such a solution, if it is noncompact, splits a Euclidean factor. In this paper we prove analogous results for the Ricci flow. 

Let $(M, g(t))$, $t \in (-\infty, T]$, be an $n$-dimensional solution of Ricci flow which is complete, nonflat, satisfies the curvature bound
    \[\sup_{t\in(-\infty,T]} \sup_M |\Rm| < \infty,\]
and is $\kappa$-noncollapsed.\footnote{That is, $\vol(B_{g(t)}(x,r)) \geq \kappa r^n$ whenever $|\Rm| \leq r^{-2}$ in $B_{g(t)}(x,r)$.} We then call $(M,g(t))$ an ancient $\kappa$-solution. All of our results concern ancient $\kappa$-solutions which are either weakly PIC2 or have nonnegative sectional curvature. An ancient $\kappa$-solution is Type~I if there is a constant $C > 0$ such that
    \[|\Rm| \leq \frac{C}{T -t}\]
for all $t < T$, and is Type~II otherwise. 

A Riemannian manifold $(M,g)$ of dimension $n \geq 4$ is said to have nonnegative isotropic curvature if 
    \[\Rm_{1313}+\Rm_{1414}+\Rm_{2323}+\Rm_{2424}-2\Rm_{1234} \geq 0\]
for all orthonormal four-frames $\{e_1,e_2,e_3,e_4\}$. If the inequality is strict, $(M,g)$ is said to have positive isotropic curvature. For short, such spaces are said to be weakly PIC, and strictly PIC, respectively. If $\mathbb{R}^2\times(M,g)$ has nonnegative (positive) isotropic curvature, then $(M,g)$ is said to be weakly (strictly) PIC2. We note that a space with nonnegative curvature operator is necessarily weakly PIC2, and a weakly PIC2 manifold has nonnegative sectional curvature (see \cite[p.~100]{Brendle_10}). The PIC condition was introduced in \cite{Micallef-Moore}, and was studied in conjunction with the Ricci flow on 4-manifolds by Hamilton \cite{Hamilton_PIC}. It is preserved by the flow in all dimensions \cite{Nguyen, Brendle-Schoen_strict}, and played an essential role in the proof of the differentiable sphere theorem \cite{Brendle-Schoen_strict, Brendle-Schoen_weak}. 

We establish the following dimension-reduction result.

\begin{theorem}\label{main}
Let $(M,g(t))$, $t \in (-\infty, T]$, be a simply connected, noncompact Type I ancient $\kappa$-solution with nonnegative sectional curvature. There is an integer $1 \leq m \leq n-2$ such that, for every $t \in (-\infty, T]$, $(M,g(t))$ is isometric to the product of $\mathbb{R}^m$ with a compact Riemannian manifold.
\end{theorem}

We note the following direct consequence of Theorem~\ref{main} (which, at least for ancient $\kappa$-solutions, answers Question 8 on p. 390 of \cite{Chow-Lu-Ni}). 

\begin{corollary}\label{Type II}
If $(M,g(t))$, $t \in (-\infty, T]$, is a noncompact ancient $\kappa$-solution with positive sectional curvature, then it is Type~II, i.e., 
    \[\limsup_{t \to -\infty} \sup_M \, (-t) |\Rm| = \infty.\]
\end{corollary}

If $(M,g(t))$ is as in Corollary~\ref{Type II}, then Hamilton's rescaling procedure \cite[Section~16]{Hamilton_Singularities} yields a sequence of flows $(M, \lambda_k g(\lambda_k^{-1} t), p_k)$ which converge smoothly (in the pointed Cheeger--Gromov sense) to an eternal solution with globally bounded scalar curvature, and which attains its maximum scalar curvature at a point at $t = 0$. If in addition $(M,g(t))$ is PIC2 and the eternal limit has positive Ricci curvature, then it is a steady soliton by Brendle's Harnack inequality \cite{Brendle_Harnack}. Examples of $\kappa$-noncollapsed steady solitons which are PIC2 include the Bryant soliton and the examples constructed in \cite{Lai}. Every known example of a noncompact, PIC2, Type~II ancient $\kappa$-solution is a steady soliton. 

A further consequence of Theorem~\ref{main} is the following classification result for Type~I ancient $\kappa$-solutions which are weakly PIC2. This was previously established in \cite{Li}, but we are able to give a more succinct proof. 

\begin{theorem}\label{einstein}
Let $(M,g(t))$, $t \in (-\infty, T]$, be a simply connected Type~I ancient $\kappa$-solution which is weakly PIC2. Then $(M,g(t))$ is isometric to $\mathbb{R}^m\times(\tilde M, \tilde g(t))$, where $0 \leq m \leq n-2$, and $(\tilde M, \tilde g(t))$ is a compact locally symmetric space with positive Ricci curvature.
\end{theorem}

The compact factor $(\tilde M, \tilde g(t))$ in Theorem~\ref{einstein} need not be a shrinking soliton, but does split as a product of Einstein manifolds. 

Refinements of Theorem~\ref{einstein} have been proven under stronger hypotheses. A nonnegatively curved Type~I ancient Ricci flow of dimension 2 has constant curvature by \cite{Daskalopoulos-Hamilton-Sesum}. A 3-dimensional Type~I ancient $\kappa$-solution with nonnegative sectional curvature has universal cover isometric to $\mathbb{S}^3$ if it is compact \cite{Perelman}, or $\mathbb{R}\times\mathbb{S}^2$ if it is noncompact (see \cite{Hallgren}, \cite{Zhang} and \cite{Brendle_20}). In higher dimensions, a compact Type~I ancient $\kappa$-solution which is strictly PIC2 has constant curvature by \cite{Brendle-Huisken-Sinestrari} (see also \cite{Ni}). A noncompact Type~I ancient $\kappa$-solution which is
uniformly PIC and weakly PIC2 has universal cover isometric to $\mathbb{R}\times\mathbb{S}^{n-1}$ (see \cite{Naff_19} and  \cite{Brendle-Naff}). Enders--M\"{u}ller--Topping showed that, for a general complete solution of Ricci flow, blow-ups at a Type~I singularitiy are shrinking solitons \cite{Enders-Mueller-Topping}. 

Let us remark on the hypotheses of Theorem~\ref{einstein}. Bakas, Ni and Kong constructed examples of compact ancient solutions of the Ricci flow on certain spheres which are Type~I but fail to be locally symmetric \cite{Bakas-Ni-Kong}. Among them are solutions which have positive curvature operator but are not $\kappa$-noncollapsed, and solutions that are $\kappa$-noncollapsed but only have positive sectional curvature. This means that Theorem~\ref{einstein} fails if we remove the $\kappa$-noncollapsing assumption, or if we replace weakly PIC2 by nonnegative sectional curvature. 

As an illustration we note that in dimension 4, Theorem~\ref{einstein} amounts to the following statement. 
\begin{corollary}
A simply connected, weakly PIC2, Type~I ancient $\kappa$-solution of dimension 4 is isometric (up to parabolic rescaling) to one of the following:
\begin{itemize}
    \item The product of two self-similarly shrinking round two-spheres. 
    \item The self-similarly shrinking solution generated by
    \[\mathbb{S}^4, \;\; \mathbb{CP}^2, \;\; \mathbb{R}\times\mathbb{S}^3, \;\; \text{or} \;\; \mathbb{R}^2\times\mathbb{S}^2.\]
\end{itemize}
\end{corollary}
\iffalse
If $(M,g(t))$ is 5-dimensional, it is either isometric to the product of $\mathbb{R}$ with one of the above, or to one of the following:
    \[\mathbb{S}^5, \;\; SU(3)/SO(3), \;\; SU(4)/Sp(2), \;\; \text{or} \;\; \mathbb{S}^2\times\mathbb{S}^3.\]
\fi

\subsection{Further related work} All convex ancient solutions of curve-shortening flow have been classified \cite{Daskalopoulos-Hamilton-Sesum_curve, BLT_curve}. Ancient solutions of mean curvature flow which are convex, interior noncollapsed and uniformly two-convex were classified in all dimensions in \cite{Brendle-Choi_higherdimensions, Angenent-Daskalopoulos-Sesum_20}. Recently, there has been progress towards a classification of all interior noncollapsed convex ancient solutions of mean curvature flow in $\mathbb{R}^4$ \cite{Zhu, Choi-Haslhofer-Hershkovits, Du-Haslhofer_a, Du-Haslhofer_b, CDDHS}. If the noncollapsing assumption is dropped, many more examples of convex ancient solutions arise \cite{BLT_collapse}. Important structural results for such solutions were established in \cite{Wang, Brendle-Naff_noncollapsing, BLL}, in addition to \cite{BLT_collapse}.  

All positively curved ancient solutions of Ricci flow of dimension 2 have been classified \cite{Daskalopoulos-Hamilton-Sesum, Chu, Daskalopoulos-Sesum}. Ancient solutions which are $\kappa$-noncollapsed, unformly PIC, and strictly PIC2 were classified in \cite{Brendle-Naff, Brendle-Daskalopoulos-Naff-Sesum}. 

Some of the most fundamental results concerning ancient solutions of the Ricci flow were established by Perelman \cite{Perelman}, and then put to use in his construction of a flow with surgeries. The surgery construction for mean curvature flow carried out in \cite{Haslhofer-Kleiner} also makes extensive use of ancient solutions (but they play less of a role in other versions of the construction \cite{Huisken-Sinestrari_surgery, Brendle-Huisken}). 

\subsection{Outline} In Section~\ref{splitting} we prove Theorem~\ref{main}. This is achieved by extracting an asymptotic shrinker at $t = -\infty$ (obtained by parabolic rescalings), and an asymptotic cone (obtained by scaling down the distance function at a fixed time), and then relating these two different limits. The asymptotic shrinker is simply connected and has nonnegative sectional curvature, so it splits a Euclidean factor by work of Munteanu and Wang. Using the Type~I property, we show that the asymptotic cone is isometric to said Euclidean factor. An application of Toponogov's triangle comparison and splitting theorems then gives Theorem~\ref{main}.

In Section~\ref{einstein section} we prove Theorem~\ref{einstein}. In light of Theorem~\ref{main}, it suffices to show that a compact Type~I ancient $\kappa$-solution is locally symmetric and has positive Ricci curvature. This follows from a standard blow-down argument and convergence results for PIC2 metrics under the Ricci flow. 

\subsection{Acknowledgements} We would like to express thanks to S. Brendle, G. Huisken, K. Naff and M. Wink for helpful conversations relating to this work. We also thank Y. Li for his correspondence, which led us to correct an error in the statement of Theorem~\ref{main}. 

%%%%%%%%%%%%%%%%%%%%%%%%%%%%%%%%%%%%%%%%%%%
%%%%%%%%%%%%%%%%%%%%%%%%%%%%%%%%%%%%%%%%%%%
%%%%%%%%%%%%%%%%%%%%%%%%%%%%%%%%%%%%%%%%%%%
%%%%%%%%%%%%%%%%%%%%%%%%%%%%%%%%%%%%%%%%%%%
%%%%%%%%%%%%%%%%%%%%%%%%%%%%%%%%%%%%%%%%%%%

\section{Dimension reduction for noncompact Type~I solutions}\label{splitting}

In this section we establish Theorem~\ref{main}. A key tool in the proof is the following result of Naber \cite[Theorem~3.1]{Naber}, which generalised earlier work of Perelman \cite{Perelman}. 
\begin{lemma}\label{shrinker}
Let $(M,g(t))$, $t \in (-\infty, T]$, be a Type~I ancient $\kappa$-solution. Fix $p \in M$ and let $\lambda_k \to 0$ be a sequence of scales. The pointed rescaled flows $(M, \lambda_kg(\lambda_k^{-1}t), p)$, $t \in (-\infty, \lambda_kT]$, subconverge smoothly\footnote{Here we mean smooth convergence in the pointed Cheeger--Gromov sense --- see \cite{Hamilton_Singularities} for the precise definition.} to a pointed gradient shrinking soliton $(\bar M, \bar g(t), \bar p)$, $t \in (-\infty, 0)$.
\end{lemma}

We refer to any shrinking soliton which arises by the rescaling procedure described in Lemma~\ref{shrinker} as an asymptotic shrinker for $(M,g(t))$. The notation $(\bar M, \bar g(t))$ will always refer to an asymptotic shrinker. Note that if $M$ is compact, then $M$ and $\bar M$ are diffeomorphic. This follows from Hamilton's distance distortion estimate (see Lemma~\ref{distance distortion} below). 

\begin{remark} The existence of an asymptotic shrinker requires that the rescaled flows $(M, \lambda_kg(\lambda_k^{-1}t), p)$ admit a uniform lower bound for the injectivity radius at $p$. This follows from the $\kappa$-noncollapsing assumption. The proofs of Theorem~\ref{main} and Theorem~\ref{einstein} only use $\kappa$-noncollapsing in this way, so it could be replaced by any other condition which leaves Lemma~\ref{shrinker} valid.
\end{remark}

We also make use of the following result. The assertion is that a space with nonnegative sectional curvature, and whose asymptotic cone splits a Euclidean factor, must itself split the same factor. This seems to be well-known, but for completeness we provide a proof based on the Toponogov splitting theorem. 

We recall that for a complete Riemannian manifold $(M,g)$ with nonnegative sectional curvature, given a sequence $\lambda_k \to 0$, the sequence of metric spaces $(M, \lambda_k d_g, p)$ subconverges in the pointed Gromov--Hausdorff sense. Any limiting metric space obtained from $(M,g)$ in this way is called an asymptotic cone for $(M,g)$ at $p$. 

\begin{lemma}\label{cone split}
    Let $(M, g)$ be a complete Riemannian $n$-manifold of nonnegative sectional curvature and suppose that it has an asymptotic cone which is isometric to $\mathbb{R}^m$. The space $(M, g)$ then splits isometrically as $\mathbb{R}^m \times \tilde M^{n-m}$, where $\tilde M^{n-m}$ is compact.
\end{lemma}

Before proceeding with the proof, we recall Toponogov's angle comparison theorem (see eg. \cite[Theorem 10.3.1]{Burago-Burago-Ivanov}), which will play a key role. 

\begin{theorem}[Toponogov]
Consider a complete Riemannian manifold $(M,g)$ which has nonnegative sectional curvature. Let $\rho:[0,S] \to (M,g)$ and $\eta:[0,T] \to (M,g)$ be length minimising unit-speed geodesics such that $\rho(0) = \eta(0)$. The function 
    \begin{equation}\label{angle}
    (s,t) \mapsto \arccos\bigg(\frac{s^2 + t^2 - d_g(\rho(s),\eta(t))^2}{2st}\bigg)
    \end{equation}
is nonincreasing in $s$ for any fixed $t \in [0,T]$. 
\end{theorem}

The quantity in \eqref{angle} converges to the angle between $\rho'(0)$ and $\eta'(0)$ as $(s,t) \to 0$. 

\begin{proof}[Proof of Lemma~\ref{cone split}]
Fix some $p\in M$ and a sequence of positive numbers $\lambda_k \to 0$, and construct an asymptotic cone by extracting a subsequential pointed Gromov--Hausdorff limit of $(M,\lambda_k d,p)$, where $d$ is the Riemannian distance function on $(M,g)$. We assume that the cone is isometric to $\mathbb{R}^m$.

Let $B$ denote the ball of radius 2 around the origin in $\mathbb{R}^m$. By the Gromov--Hausdorff convergence, there are maps $\phi_k : B \to M$ such that $\phi(0) = p$ and
    \[\lim_{k\to \infty} \big| d_{\mathbb{R}^m}(x,y) - \lambda_k d(\phi_k(x), \phi_k(y))\big| = 0\]
for all $x, y \in B$. In particular, for the points
    \[x_k := \phi_k((1,0,...,0)), \qquad y_k := \phi_k((-1,0,...,0)),\]
we have
    \begin{equation}\label{cone split asymptotics}
    \lim_{k \to \infty}\lambda_k d(p, x_k) = 1, \qquad \lim_{k \to \infty}\lambda_k d(p, y_k) = 1, \qquad \lim_{k \to \infty} \lambda_k d(x_k, y_k) = 2 .
    \end{equation}
Since both $x_k$ and $y_k$ escape to infinity in $(M,g)$ as $k \to \infty$, we can construct geodesic rays $\gamma_x$ and $\gamma_y$ emanating from $p$ as the limit of geodesic segments from $p$ to $x_k$ and $y_k$, respectively. If we denote by $\alpha_k \in [0,\pi]$ the angle of the geodesic triangle $(p, x_k, y_k)$ at $p$, Toponogov's angle comparison theorem implies that $\alpha_k \geq \theta_k$, where
    \[\theta_k := \arccos \frac{d(p, x_k)^2 + d(p, y_k)^2 - d(x_k, y_k)^2}{2d(p,x_k)d(p, y_k)}\]
is the corresponding Euclidean comparison angle. Inserting \eqref{cone split asymptotics}, we obtain
    \[\lim_{k\to \infty} \theta_k = \lim_{k\to\infty} \arccos\bigg( \frac{d(p, x_k)^2 + d(p, y_k)^2 - d(x_k, y_k)^2}{2d(p,x_k)d(p, y_k)} \bigg) = \pi,\]
and hence 
    \[\lim_{k \to \infty} \alpha_k = \pi.\]
We conclude that the angle formed by the rays $\gamma_x$ and $\gamma_y$ at $p$ is $\pi$, and hence the concatenation $\gamma := -\gamma_x \frown \gamma_y : \mathbb{R} \to M$ is smooth (by Picard--Lindel\"{o}f). 

We now claim that every subinterval of $\gamma$ is length-minimizing between its endpoints. To see this, consider a new pair of sequences $\Bar{x}_k \in \gamma_x$ and $\Bar{y}_k \in \gamma_y$, defined uniquely by the requirement that
    \[d(p, x_k)=d(p, \bar x_k), \qquad d(p, y_k)=d(p, \bar y_k).
    \]
We appeal to the angle comparison theorem again to control $d(x_k,\bar x_k)$ and $d(y_k,\bar y_k)$. From the definition of $\gamma_x$ and $\gamma_y$, the angles at $p$ formed by the geodesic triangles $(p, x_k ,\bar x_k)$ and $(p, y_k,  \bar y_k)$ converge to zero, so by angle comparison we have
    \[ 0 \geq \lim_{k\to\infty} \arccos \bigg(1 - \frac{\lambda_k^2d(\bar x_k, x_k)^2}{2\lambda_k^2 d(p,x_k)d(p,\bar x_k)}\bigg),\]
and hence \eqref{cone split asymptotics} implies
    \[\lim_{k \to \infty}\lambda_k d(\bar x_k, x_k) = 0.\] 
Similarly,
    \[\lim_{k \to \infty}\lambda_k d(\bar y_k, y_k) = 0.\]
We now multiply the inequality
    \[d(x_k, y_k) - d(x_k, \bar x_k) - d(\bar y_k, y_k) \leq d(\bar x_k, \bar y_k) \leq d(\bar x_k, x_k) + d(x_k, y_k) + d(y_k, \bar y_k),
    \]
by $\lambda_k$ and take the limit to obtain
    \[\lim_{k \to \infty} \lambda_k d(\bar x_k, \bar y_k) = 2.
    \]
We thus conclude that $\bar \theta_k \to \pi$, where $\bar \theta_k$ is the Euclidean comparison angle
    \[\bar \theta_k = \arccos\bigg(\frac{d(p, \bar x_k)^2 + d(p, \bar y_k)^2 - d(\bar x_k, \bar y_k)^2}{2d(p,\bar x_k)d(p, \bar y_k)}\bigg).
    \]
By angle monotonicity, we conclude that $\bar \theta_k = \pi$ for all $k$, and hence
    \[d(\bar x_k, \bar y_k) = d(\bar x_k, p) + d(p, \bar y_k)\]
for all $k$. It follows that $\gamma$ is minimizing on every subinterval.

We now appy Toponogov's splitting theorem to conclude that $(M,g)$ is isometric to $\mathbb{R} \times (\tilde M^{n-1}, \tilde g)$, where $(\tilde M^{n-1}, \tilde g)$ is a complete Riemannian $(n-1)$-manifold with nonnegative sectional curvature. The space $(\tilde M^{n-1}, \tilde g)$ has an asymptotic cone isometric to $\mathbb{R}^{m-1}$, so we can repeat the above argument, and continue iteratively until we have exhibited $(M,g)$ as $\mathbb{R}^m \times (\tilde M^{n-m}, \tilde g)$, where $(\tilde M^{n-m}, \tilde g)$ is a Riemannian $(n-m)$-manifold whose asymptotic cone is a single point. It follows that $\tilde M^{n-m}$ is compact, for otherwise it would contain a ray, in which case all of its asymptotic cones would also contain a ray. 
\end{proof}

We recall that one can compare the Riemannian distance at different times along the Ricci flow using Hamilton's distance distortion estimate. For a Type I ancient solution with nonnegative Ricci curvature we have the following. 

\begin{lemma}\label{TypeIdistance}
    Let $(M,g(t))$, $t \in (-\infty, 0]$, be an ancient solution of the Ricci flow with nonnegative Ricci curvature and which satisfies the Type~I condition $|\Rm| \leq \bar C(-t)^{-1}$, $t < 0$. There is a constant $C=C(n, \bar C)$ such that
        \[0 \leq d_{g(s)}(x, y) - d_{g(t)}(x, y) \leq C(\sqrt{-s} - \sqrt{-t})
        \]
    for every $x, y \in M$ and $-\infty < s < t \leq 0$.
\end{lemma}
\begin{proof}
Since $\Ric \geq 0$, we have 
    \[\frac{d}{dt} d_{g(t)}(x,y) \leq 0.\]
This gives the first inequality.

Set $\sigma = (\sup_M |\Rm|(t))^{-1/2}$ in part (a) of \cite[Lemma~17.4]{Hamilton_Singularities}, and combine the result with \cite[Lemma~17.3]{Hamilton_Singularities} to obtain 
    \[\frac{d}{dt} d_{g(t)}(x,y) \geq -\frac{C}{\sqrt{-t}}.\]
After integrating, this gives the second inequality.
\end{proof}

Before carrying out the proof of Theorem~\ref{main}, we note the following consequence of Lemma~\ref{TypeIdistance}. 

\begin{lemma}\label{simply connected}
Let $(M,g(t))$, $t \in (-\infty,T]$, be a Type~I ancient $\kappa$-solution with nonnegative sectional curvature. Let $(\bar M, \bar g(t))$, $t \in (-\infty,0)$, be an asymptotic shrinker for $(M,g(t))$. If $M$ is simply connected, then $\bar M$ is simply connected.
\end{lemma}

\begin{proof}
Suppose $M$ is simply connected. Let $\Gamma \subset \bar M$ be the image of a continuous map from $S^1$ into $\bar M$. We claim that $\Gamma$ can be contracted to a point in $\bar M$ via a continuous homotopy. 

By definition, there is a point $p \in M$, a point $\bar p \in \bar M$, and a sequence of scales $\lambda_k \to 0$ such that $(M, \lambda_k g(\lambda_k^{-1}t), p)$ converges to $(\bar M, \bar g(t), \bar p)$ in the pointed Cheeger--Gromov sense. Let us write $g_k(t) := \lambda_k g(\lambda_k^{-1} t)$. In particular, there are maps $\Phi_k : B_{\bar g(-1)}(\bar p, k) \to M$ with the following properties:
	\begin{itemize}
		\item Each $\Phi_k$ is a diffeomorphism onto its image, and $\Phi_k(\bar p) = p$.
		\item The pulled back metrics $\Phi_k^* g_k(-1)$ converge locally smoothly to $\bar g(-1)$ as $k \to \infty$.
	\end{itemize}
	
For each $k$, we define $\Gamma_k := \Phi_k(\Gamma)$. We may choose $R > 0$ so that 
	\[\Gamma \subset B_{\bar g(-1)}(\bar p, R-1).\]
It follows that there is an integer $m$ such that 
	\begin{equation}%\label{simply connected 1}
		\Gamma_k \subset B_{g_k(-1)}(p, R)
	\end{equation}
for all $k \geq m$. Also, making $m$ larger if necessary, we can ensure that 
    \[B_{g_k(-1)}(p, R) \subset \Phi_k\Big(B_{\bar g(-1)}(\bar p, R+1)\Big) \subset B_{g_m(-1)}(p,R+2)\]
for all $k \geq m$. We thus have
	\[\Gamma_k \subset B_{g_m(-1)}(p,R+2)\]
for all $k \geq m$. It will be important that the right-hand side is independent of $k$. 

By the soul theorem \cite{Cheeger-Gromoll_soul}, there is a compact totally convex submanifold $\Sigma$ of $(M, g_m(-1))$ such that $M$ is diffeomorphic to the normal bundle of $\Sigma$. Let us denote the normal bundle of $\Sigma$ by $N$, and identify $N$ with $M$. Observe that, by scaling in the fiber, we obtain a deformation retraction of $N$ onto $\Sigma$. Since $M = N$ is assumed to be simply connected, this implies that $\Sigma$ is simply connected.

Let $N'$ be the subbundle of $N$ consisting of normal vectors whose length with respect to $g_m(-1)$ is at most $E$. If $E$ is chosen large enough, then 
    \[\Gamma_k \subset B_{g_m(-1)}(p,R+2) \subset N',\]
for all $k \geq m$. Therefore, we can contract $\Gamma_k$ to a point in $N'$ via a continuous homotopy, by first retracting onto $\Sigma$, and then using that $\Sigma$ is simply connected to contract to a point in $\Sigma$. Choosing $R'>R+2$ large enough so that 
	\[N' \subset B_{g_m(-1)}(p,R'),\]
we ensure that each $\Gamma_k$ can be contracted to a point in $B_{g_m(-1)}(p,R')$. 

We have 
	\[B_{g_m(-1)}(p, R') = B_{g(-\lambda_m^{-1})}(p, R' / \sqrt{\lambda_m}).\]
By the Type~I distance distortion estimate (Lemma~\ref{distance distortion}),
	\[B_{g(-\lambda_m^{-1})}(p, R' / \sqrt{\lambda_m}) \subset B_{g(t)}(p, R' / \sqrt{\lambda_m} + C\sqrt{-t})\]
for all $t \leq -\lambda_m^{-1}$, so for $\lambda_k \leq \lambda_m$ we have 
	\[B_{g(-\lambda_m^{-1})}(p, R' / \sqrt{\lambda_m}) \subset B_{g(-\lambda_k^{-1})}(p, R'/\sqrt{\lambda_m} + C/\sqrt{\lambda_k}).\]
	It follows that 
	\[B_{g(-\lambda_m^{-1})}(p, R'/\sqrt{\lambda_m}) \subset B_{g_k(-1)}(p, R'\sqrt{\lambda_k}/\sqrt{\lambda_m} + C),\]
	and hence 
	\[B_{g_m(-1)}(p, R') \subset B_{g_k(-1)}(p, R' + C)\]
whenever  $\lambda_k \leq \lambda_m$. In particular, for all sufficiently large $k$, each of the loops $\Gamma_k$ can be contracted to a point in $B_{g_k(-1)}(p, R' + C)$. But for sufficiently large $k$ the image of $B_{\bar g(-1)}(\bar p, k)$ under $\Phi_k$ contains $B_{g_k(-1)}(p, R' + C)$. Since $\Phi_k$ is a diffeomorphism onto its image, for sufficiently large $k$, we can take any homotopy of $\Gamma_k$ to a point in $B_{g_m(-1)}(p, R')$ and pull back by $\Phi_k$. This yields a homotopy of $\Gamma$ to a point in $\bar M$. We conclude that $\bar M$ is simply connected.  
\end{proof}

We are now ready to prove Theorem~\ref{main}. 

\begin{proof}[Proof of Theorem~\ref{main}]
For simplicity, we assume $T = 0$. This can be arranged by shifting the time variable, and so does not constitute a loss of generality. 

Fix $p \in M$ and a sequence of positive numbers $\lambda_k \to 0$. Let $g_k(t):=\lambda_k g(\lambda_k^{-1}t)$. Appealing to Lemma~\ref{shrinker}, after passing to a subsequence if necessary, we can extract an asymptotic shrinker $(\bar M, \bar g(t), \bar p)$, $t\in(-\infty,0)$, as a limit of the sequence $(M, g_k(t), p)$.
For each $R>0$ and $t < 0$ we have
    \[\lim_{k \to \infty} d_{GH}\Big(B_{g_k(t)}(p,R), B_{\bar g(t)}(\bar p, R)\Big) = 0
    \]
where $d_{GH}$ stands for the pointed Gromov--Hausdorff distance. Here, and for the remainder of the proof, all balls are closed.

After passing to a further subsequence if necessary, we may assume the sequence $(M,\sqrt{\lambda_k}d_{g(-1)}, p)$ converges in the pointed Gromov--Hausdorff sense to an asymptotic cone for $(M,g(-1))$, which we denote $(\hat M, \hat d, \hat p)$. For each $R>0$ we write $\hat B(R)$ for the closed metric ball of radius $R$ about $\hat p$ in $\hat M$. We then have 
    \[\lim_{k \to \infty} d_{GH}\Big(B_{\lambda_k g(-1)}(p,R), \hat B(R)\Big)= 0\]
for every $R>0$.
    
Appealing to Lemma~\ref{simply connected}, we conclude that the asymptotic shrinker $(\bar M, \bar g(t))$ is simply connected. Since $(\bar M, \bar g(t))$ is also noncompact and has nonnegative sectional curvature, by a result of Munteanu--Wang \cite[Corollary 3]{Munteanu-Wang}, it splits a Euclidean factor $\mathbb{R}^m$, where $m \in \{1,\dots,n\}$. We will use this fact to establish that $\hat M$ is isometric to $\mathbb{R}^m$. 

Since $d_{g_k(t)} = \sqrt{\lambda_k}d_{g(\lambda_k^{-1}t)}$, we can use Lemma~\ref{TypeIdistance} to compare  $B_{\lambda_k g(-1)}(p,R)$ and $B_{g_k(t)}(p, R)$, as follows. For $x$ and $y$ in $M$ we have
\begin{align}\label{distance distortion}
        |d_{g_k(t)}(x, y) - \sqrt{\lambda_k}\,d_{g(-1)}(x,y)| &=  \sqrt{\lambda_k}\,|d_{g(\lambda_k^{-1}t)}(x, y) - d_{g(-1)}(x,y)| \notag \\
        & \leq \sqrt{\lambda_k}\,C\,(\sqrt{-\lambda_k^{-1}t} + 1) \notag \\
        &\leq C(\sqrt{-t} + \sqrt{\lambda_k}).
\end{align}
For any fixed $R>0$ and $t < 0$, Lemma~\ref{TypeIdistance} also tells us that 
    \[B_{g_k(t)}(p, R) \subset B_{\lambda_k g(-1)}(p, R)\]
whenever $k$ is so large that $\lambda_k^{-1} t < -1$. On the other hand \eqref{distance distortion} implies
    \[B_{\lambda_k g(-1)}(p,R) \subset B_{g_k(t)}(p, R + C(\sqrt{-t} + \sqrt{\lambda_k})).\]
Combining these two facts we obtain
    \[d_{GH}\Big(B_{\lambda_k g(-1)}(p,R) , B_{g_k(t)}(p, R)\Big) \leq C(\sqrt{-t} + \sqrt{\lambda_k}).\]

We now apply the triangle inequality to estimate
    \begin{align*}
    d_{GH}\Big(&B_{g_k(t)}(p, R), \hat B(R)\Big)\\
    &\leq d_{GH}\Big(B_{g_k(t)}(p, R), B_{\lambda_kg(-1)}(p,R)\Big) + d_{GH}\Big(B_{\lambda_kg(-1)}(p,R), \hat B(R)\Big) \\
    &\leq C(\sqrt{-t} + \sqrt{\lambda_k}) + d_{GH}\Big(B_{\lambda_kg(-1)}(p,R), \hat B(R)\Big).
    \end{align*}
Combining this with
    \begin{align*}
        d_{GH}\Big(&B_{\bar g(t)}(\bar p, R), \hat B(R)\Big)\\
        &\leq d_{GH}\Big(B_{\bar g(t)}(\bar p, R), B_{g_k(t)}(p, R)\Big) + d_{GH}\Big(B_{g_k(t)}(p, R), \hat B(R)\Big),
    \end{align*}
and taking the limit as $k \to \infty$, we obtain
    \[d_{GH}\Big(B_{\bar g(t)}(\bar p, R), \hat B(R)\Big) \leq C\sqrt{-t}\]
for every $R>0$ and $t < 0$. It follows that, for a sequence $t_j \to 0$, the pointed Gromov--Hausdorff limit $(\bar M, \bar g(-t_j), \bar p)$ is $(\hat M, \hat d, \hat p)$. Recalling that $(\bar M, \bar g(t))$ splits as the product of $\mathbb{R}^m$ with a compact factor, we see that $(\hat M, \hat d)$ is isometric to $\mathbb{R}^m$. 

We now appeal to Lemma~\ref{cone split}. This shows that $(M, g(t))$ is isometric to the product of $\mathbb{R}^m$ with a compact Riemannian manifold at each time. The cases $m = n$ and $m = n - 1$ are ruled out by the assumption that $(M,g(t))$ is nonflat.
\end{proof}

%%%%%%%%%%%%%%%%%%%%%%%%%%%%%%%%%%%%%%%%%%%
%%%%%%%%%%%%%%%%%%%%%%%%%%%%%%%%%%%%%%%%%%%
%%%%%%%%%%%%%%%%%%%%%%%%%%%%%%%%%%%%%%%%%%%
%%%%%%%%%%%%%%%%%%%%%%%%%%%%%%%%%%%%%%%%%%%
%%%%%%%%%%%%%%%%%%%%%%%%%%%%%%%%%%%%%%%%%%%

\section{Type~I solutions are shrinking solitons}\label{einstein section}

In this section we prove Theorem~\ref{einstein}. Although this result was previously established by Li in \cite{Li}, using Theorem~\ref{main}, we are able to give a shorter proof. Let $(M,g(t))$, $t \in (-\infty,T]$, be a simply connected Type~I ancient $\kappa$-solution which is weakly PIC2. By Theorem~\ref{main}, we know that if $(M,g(t))$ is noncompact, then it splits as the product of a Euclidean factor with a compact solution. We claim that the compact factor is locally symmetric and has positive Ricci curvature. In fact, we prove the following stronger statement.

\begin{proposition}\label{compact einstein}
Let $(M,g(t))$, $t \in (-\infty, T]$, be a compact Type~I ancient $\kappa$-solution which is weakly PIC2. Then $(M,g(t))$ is a locally symmetric space, and its universal cover splits as a product of compact Einstein manifolds with positive Ricci curvature. 
\end{proposition}

We first establish local symmetry. This is achieved by a standard argument based on the Berger holonomy theorem. After reducing to the case of a simply connected irreducible solution, one can argue that $(M,g(t))$ is either locally symmetric or else has an asymptotic shrinker isometric to $\mathbb{S}^n$ or $\mathbb{CP}^n$. In the latter cases, the solution itself is isometric $\mathbb{S}^n$ or $\mathbb{CP}^n$, respectively. For the sake of completeness we provide a detailed account of these arguments (versions of which have appeared in \cite{Ni, Deng-Zhu, Li}). 

We will make use of the following lemma.

\begin{lemma}\label{sphere}
If $(M,g)$ is a symmetric space which is homeomorphic to $\mathbb{S}^n$ then, up to scaling, $(M,g)$ is isometric to $\mathbb{S}^n$.
\end{lemma}
\begin{proof}
By a result of Kostant \cite{Kostant}, since $M$ is a topological sphere, the holonomy group of $(M,g)$ is $SO(n)$. Since the curvature tensor of $(M,g)$ is parallel, this immediately implies that $(M,g)$ has constant curvature. 
\end{proof} 

\begin{proposition}\label{compact symmetry}
Let $(M,g(t))$, $t \in (-\infty, T]$, be a compact Type~I ancient $\kappa$-solution which is weakly PIC2. Then $(M,g(t))$ is a locally symmetric space.
\end{proposition}
\begin{proof}
We need only consider the case that $(M,g(t))$ is simply connected. If not, we pass to the universal cover. This may be noncompact, but it splits as the product of a Euclidean space with a compact factor by Theorem~\ref{main}, in which case it suffices to show that the compact factor is locally symmetric. 

In dimensions 2 and 3, $(M,g(t))$ has constant curvature by \cite{Daskalopoulos-Hamilton-Sesum} and \cite{Perelman}, so let us assume $n \geq 4$. We may also assume without loss of generality that $(M, g(t))$ is irreducible. Indeed, there exists a time $\bar t$ such that $(M, g(t))$ is isometric to a fixed number of irreducible factors for all $t \leq \bar t$. But if $(M,g(t))$ is locally symmetric for $t \leq \bar t$, then this remains true up to $t = T$. 

Consider a sequence of times $t_k \to -\infty$. For each $k$ we apply the Berger holonomy theorem (see eg. \cite[Corollary 10.92]{Besse}). This tells us that we are in one of the following situations, possibly after passing to a subsequence in $k$.

\emph{Case 1.} $(M, g(t_k))$ is a locally symmetric space for all $k$, and hence $(M, g(t))$ is locally symmetric for all $t \in (-\infty, T]$. 

\emph{Case 2.} The holonomy group of $(M, g(t_k))$ is not $SO(n)$ or $U(n/2)$. In this case $(M, g(t_k))$ is Einstein, and is therefore locally symmetric by work of Brendle \cite{Brendle_Einstein}.

\emph{Case 3.} The holonomy group of $(M, g(t_k))$ is $SO(n)$. In this case, by the strong maximum principle proven in \cite[Proposition~10]{Brendle-Schoen_weak}, we conclude that $(M, g(t_k))$ is strictly PIC2, and hence is diffeomorphic to a sphere by \cite{Brendle-Schoen_strict}. At this point we can apply \cite[Theorem 16]{Brendle-Huisken-Sinestrari} to conclude that $(M, g(t))$ has constant curvature for all $t \in (-\infty, T]$. Note the argument in \cite{Brendle-Huisken-Sinestrari} assumes $M$ is even-dimensional, but only in order obtain a lower bound for the injectivity radius. In our case, such a bound follows from the $\kappa$-noncollapsing.

Another way to argue is as follows. Let $(\bar M, \bar g)$ be an asymptotic shrinker for $(M, g(t))$. Once again, by the Berger theorem and \cite[Proposition 10]{Brendle-Schoen_weak}, $(\bar M, \bar g)$ is either strictly PIC2, or else is a symmetric space. In both cases $(\bar M, \bar g)$ has constant curvature, by \cite{Brendle-Schoen_strict} and Lemma~\ref{sphere}, respectively. It follows that, after rescaling, $(M,g(t_k))$ has constant curvature in the limit as $k \to \infty$, and hence $(M,g(t))$ has constant curvature for all $t \in (-\infty, T]$ by the pinching construction in \cite{Boehm-Wilking} (in fact, in this last step Huisken's pinching estimate suffices \cite{Huisken}).

\emph{Case 4.} We have $n = 2m$ and the holonomy group of $(M, g(t_k))$ is $U(m)$. That is, $(M, g(t_k))$ is a K\"{a}hler manifold. In this case we can conclude that $(M,g(t))$ is isometric to $\mathbb{CP}^m$ (we note that an argument similar to the following appeared in \cite{Deng-Zhu}). Indeed, since $(M, g(t_k))$ is weakly PIC2, and therefore has nonnegative sectional curvature, a result of Mok \cite{Mok} implies it is either biholomorphic to $\mathbb{CP}^m$, or else is isometric to a Hermitian symmetric space. In the latter case we are done, so suppose $(M, g(t_k))$ is biholomorphic to $\mathbb{CP}^m$. 

Let $(\bar M, \bar g)$ be an asymptotic shrinker for $(M,g(t))$. We know that $(\bar M, \bar g)$ is K\"{a}hler and is diffeomorphic to $\mathbb{CP}^m$, but such a space is biholomorphic to $\mathbb{CP}^m$ by work of Hirzebruch--Kodaira \cite{Hirzebruch-Kodaira} and Yau \cite{Yau}. Since the Futaki invariant of $\mathbb{CP}^m$ vanishes, we have that $(\bar M, \bar g)$ is a K\"{a}hler--Ricci soliton with vanishing Futaki invariant, and hence $(\bar M, \bar g)$ is K\"{a}hler--Einstein (see eg. \cite[p. 125]{Cao}). Using the uniqueness theorem for K\"{a}hler--Einstein metrics due to Bando and Mabuchi \cite{Bando-Mabuchi}, we conclude that $(\bar M, \bar g)$ is biholomorphically isometric to $\mathbb{CP}^m$. Therefore, after rescaling, we have that $(M, g(t_k))$ converges smoothly to $\mathbb{CP}^m$ as $k \to \infty$. By the convergence result of Chen--Tian \cite{Chen-Tian}, $(M,g(t))$ can be extended to a maximal solution which converges (up to rescaling) to $\mathbb{CP}^m$ forward in time. By Perelman's monotonicity formula for the reduced volume \cite{Perelman}, it follows that the reduced volume of $(M,g(t))$ is constant (equal to that of the shrinking soliton generated by $\mathbb{CP}^m$ --- we refer to \cite[Theorem~2.1]{Naber} for a detailed proof). Therefore, $(M,g(t))$ is a shrinking soliton, and hence is isometric to $\mathbb{CP}^m$. 

In every case, we have found that $(M,g(t))$ is a locally symmetric space. 
\end{proof}

We continue with the proof of Proposition~\ref{compact einstein}. It remains to prove that the universal cover of $(M,g(t))$ is isometric to a product of Einstein manifolds with positive Ricci curvature. This follows from some basic facts about symmetric spaces. Recall that a locally symmetric Riemannian manifold which is also simply connected is a symmetric space (see eg. \cite[Theorem~10.3.2]{Petersen}). Such a space splits as a product of simply connected irreducible factors, each of which is automatically Einstein (see eg. \cite[p. 386]{Petersen}). A simply connected irreducible symmetric space is either of compact type, Euclidean type, or noncompact type. Spaces of compact type have nonnegative curvature operator and positive Ricci curvature, spaces of Euclidean type are flat, and spaces of noncompact type have nonpositive curvature operator and negative Ricci curvature. For Cartan's classification of symmetric spaces, see eg. \cite[Chapter~X]{Helgason}.

\begin{proof}[Proof of Proposition~\ref{compact einstein}]
Let $(\bar M, \bar g)$ be an asymptotic shrinker for $(M,g(t))$. We know that $(\bar M, \bar g)$ is compact and locally symmetric (because of Proposition~\ref{compact symmetry}), so the shrinking soliton equation
    \[\bar \Ric + \bar \nabla^2 f = \lambda \bar g\]
implies that $f$ is constant, and hence $(\bar M, \bar g)$ is Einstein. Since $(M,g(t))$ is nonflat, so is $(\bar M, \bar g)$ (this follows from Perelman's monontonicity formula for the reduced volume --- see \cite[Theorem~2.1]{Naber}), and hence $\lambda > 0$. Using the Bonnet--Myers theorem, we conclude that the universal cover of $\bar M$ is compact. But $M$ and $\bar M$ are diffeomorphic, so the universal cover of $M$ is compact. 

Since $(M,g(t))$ is locally symmetric, we conclude that its universal cover $(\tilde M, \tilde g(t))$ is a compact symmetric space. There is a time $\bar t \leq T$ such that
    \begin{equation}\label{compact split}
        (\tilde M, \tilde g(t)) \cong (M_1, g_1(t)) \times \dots (M_k, g_k(t))
    \end{equation}
for all $t \leq \bar t$, where each of the factors $(M_i,g_i(t))$ is irreducible. Each of these has nonnegative sectional curvature, and may not be flat (by the $\kappa$-noncollapsing assumption). Therefore, each factor is a symmetric space of compact type, and is hence Einstein with positive Ricci curvature. By uniqueness of compact solutions, each $(M_i, g_i(t))$ extends smoothly to $t = T$, and \eqref{compact split} persists for all $t \leq T$. 
\end{proof}

\iffalse
\begin{lemma}\label{einstein lemma}
Let $(M,g(t))$, $t \in (-\infty, T]$, be a compact ancient solution of Ricci flow such that $(M,g(t))$ is Einstein for each $t \in (-\infty,T]$. We then have
    \[g(t) = \bigg(1+2\frac{R(T)}{n}(T-t)\bigg)g(T).\]
\end{lemma}
\begin{proof}
For each $t \in (-\infty, T]$, since $(M,g(t))$ is Einstein, its scalar curvature $R$ is constant and we have
    \[\Ric = \frac{R}{n}g.\]
It follows that
    \[\frac{\partial R}{\partial t} = \Delta R + 2|\Ric|^2 = \frac{2}{n}R^2,\]
so for any pair of times $t$ and $s$, 
    \[R(t)^{-1} - R(s)^{-1} = \frac{2}{n}(s - t).\]
From this we obtain
    \[\frac{\partial g}{\partial t} = -2\Ric = -2\bigg(\frac{n}{R(T)}+ 2(T-t)\bigg)^{-1} g,\]
which gives the claim upon integrating. 
\end{proof}
\fi

\bibliographystyle{abbrv}
\bibliography{references}

\end{document}